\documentclass{article}
\usepackage{amssymb,amsmath,amsthm,graphicx}
\usepackage{amsfonts}
\newtheorem{theorem}{Theorem}

\theoremstyle{definition}
\newtheorem{example}{Example}
\newtheorem{definition}{Definition}
\newtheorem{remark}{Remark}
\date{}

\title{\Large \textbf{Multi-tribrackets}}

\author{
Sam Nelson\footnote{Email: Sam.Nelson@cmc.edu. Partially supported by Simons Foundation collaboration grant $\#316709$.}\and
Evan Pauletich\footnote{Email: EPauletich19@students.claremontmckenna.edu}}

\begin{document}
\maketitle

\begin{abstract} 
We introduce \textit{multi-tribrackets}, algebraic structures for region
coloring of diagrams of knots and links with different operations at different
kinds of crossings. In particular we consider the case of \textit{component
multi-tribrackets} which have different tribracket operations at 
single-component crossings and multi-component crossings. We provide examples
to show that the resulting counting invariants can distinguish links which are
not distinguished by the counting invariants associated to the standard 
tribracket coloring. We 
reinterpret the results of \cite{NP} in terms of multi-tribrackets and consider 
future directions for multi-tribracket theory.
\end{abstract}

\parbox{5.5in} {\textsc{Keywords:} Niebrzydowski tribrackets, enhancements,
oriented knot and link invariants, link tribrackets

\smallskip

\textsc{2010 MSC:} 57M27, 57M25}

\section{Introduction}

\textit{Niebrzydowski Tribrackets}, also known as \textit{knot-theoretic 
ternary quasigroups}, are ternary operations on sets which satisfy
conditions motivated by the Reidemeister moves in knot theory. A set $X$ with
a ternary operation in which the actions of all three variables are invertible
is called a \textit{ternary quasigroup}, and a ternary quasigroup is 
\textit{knot-theoretic} if it satisfies a condition related to the 
Reidemeister III move. 

In \cite{MN}, Niebrzydowski considered region colorings of knot diagrams by
ternary quasigroups. In \cite{MN2} he introduced a (co)homology theory for
these objects analogous to quandle (co)homology and further developed this
theory in \cite{MN3}\ and \cite{MN4}. An isomorphism between this 
(co)homology theory and a cohomology theory of objects called 
\textit{local biquandles} was established in \cite{NOO} by the first author 
and collaborators.

In \cite{dsn2} \textit{biquasiles}, algebraic structures defined by pair of
quasigroups satisfying certain conditions related to Reidemeister III moves 
interpreted in terms of \textit{dual graph diagrams}, were introduced. In 
\cite{cnn} the biquasile coloring invariant was enhanced with Boltzmann weights,
which can be understood as special cases of tribracket cocycle invariants
defined in \cite{MN2}. In \cite{KN} biquasile colorings and Boltzmann 
enhancements were used to study orientable surface-links.
 
In \cite{NP}, tribracket colorings were extended to virtual knots and links
and in \cite{GNT}, tribrackets were enhanced with partial products to define
\textit{Niebrzydowski algebras}, algebraic structures for coloring $Y$-oriented
spatial graphs and handlebody-links. 

In this paper we define \textit{multi-tribrackets}, a generalization of the 
tribracket structure which 
includes virtual tribrackets from \cite{NP} as a special case and which defines
stronger counting invariants for virtual knots and links with multiple 
components. The paper is organized as follows. In Section \ref{T} we review
tribrackets. In Section \ref{M} we introduce multi-tribrackets and identify
special cases for use with virtual knots and with multicomponent links.
We compute examples of counting invariants associated to these structures
and show that they are stronger than the single tribracket counting
invariants. We end in Section \ref{Q} with questions for future research.

\section{Tribrackets}\label{T}

We begin with a definition; see \cite{MN,NOO} etc. for more.

\begin{definition}
Let $X$ be a set. A \textit{horizontal tribracket} operation on
$X$ is a ternary operation $[,,]:X\times X\times X\to X$ 
satisfying
\begin{itemize}
\item[(i)] In the equations $[a,b,c]=d$,
any three of the variables determines the fourth, and
\item[(ii)] For all $a,b,c,d\in X$
\[
 [c,[a,b,c],[a,c,d]] 
=[b,[a,b,c],[a,b,d]] 
=[d,[a,b,d],[a,c,d]]. 
\]
\end{itemize}
\end{definition}

\begin{remark} Axiom (i) can be rephrased in a few ways. One is that
for every $x,y,z\in X$ there exist unique $a,b,c\in X$ such that
\[[x,y,a]=z,\quad [x,b,y]=z\quad \mathrm{and}\quad[c,x,y]=z.\]
Equivalently, for all $x,y\in X$, the maps
$\alpha_{x,y}, \beta_{x,y},\gamma_{x,y}:X\to X$ defined by
\[\alpha_{x,y}(z) = [x,y,z],\quad \beta_{x,y}(z)=[x,z,y]\quad
\mathrm{and}\quad \gamma_{x,y}(z)=[z,x,y]\]
are invertible.
\end{remark}

\begin{example}
Let $G$ be a group. Then the map $[,,]:G\times G\times G\to G$ 
defined by $[x,y,z]=yx^{-1}z$ is
a horizontal tribracket known as a \textit{Dehn tribracket}.
Let us illustrate axiom (ii):
\[ [c,[a,b,c],[a,c,d]] =[c,ba^{-1}c,ca^{-1}d]=ba^{-1}cc^{-1}ca^{-1}d=ba^{-1}ca^{-1}d
\]
while
\[[b,[a,b,c],[a,b,d]] =[b,ba^{-1}c,ba^{-1}d]=ba^{-1}cb^{-1}ba^{-1}d=ba^{-1}ca^{-1}d\]
and
\[[d,[a,b,d],[a,c,d]]=[d,ba^{-1}d,ca^{-1}d]=ba^{-1}dd^{-1}ca^{-1}d=ba^{-1}ca^{-1}d\]
as required.
\end{example}

\begin{example}
Let $X$ be a module over the ring $\mathbb{Z}[x^{\pm 1}, y^{\pm 1}]$ of 
two-variable Laurent polynomials. Then the map $[,,]:X\times X\times X\to X$
defined by $[a,b,c]=xb+yc-xya$ is a horizontal tribracket known as an
\textit{Alexander tribracket.}
\end{example}

\begin{example}
For a finite set $X=\{1,2,\dots, n\}$ we can specify a tribracket operation
with an \textit{operation 3-tensor}, i.e., an ordered list of $n$ $n\times n$
matrices, where the entry in matrix $i$ row $j$ column $k$ is $[i,j,k]$.
For example, there are two horizontal tribracket maps on $X=\{1,2\}$, which
are specified by the operation 3-tensors
\[
\left[
\left[\begin{array}{rr}1 & 2 \\2 & 1 \end{array}\right],
\left[\begin{array}{rr}2 & 1 \\ 1 & 2\end{array}\right]
\right]\quad\mathrm{and}\quad
\left[
\left[\begin{array}{rr}2 & 1 \\ 1 & 2\end{array}\right],
\left[\begin{array}{rr}1 & 2 \\2 & 1 \end{array}\right]
\right].
\]
\end{example}

The horizontal tribracket axioms are motivated by the following 
region coloring rules:
\[\includegraphics{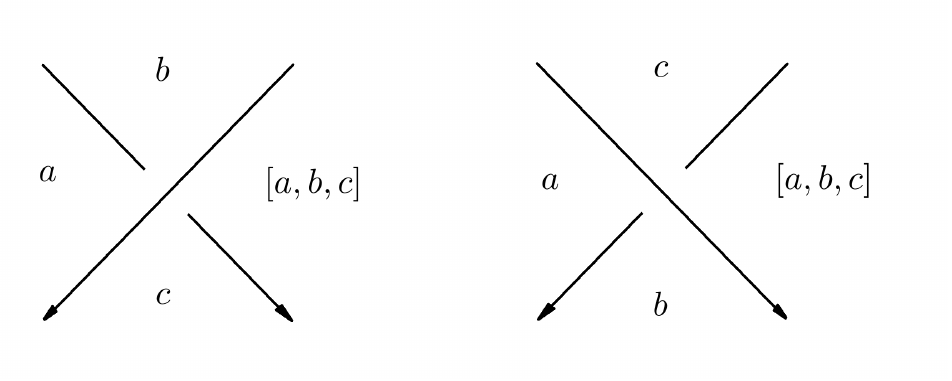}\]
and

An assignment of elements of $X$ to the regions in an oriented link diagram
$L$ satisfying these rules will be called a \textit{tribracket coloring}
or an \textit{$X$-coloring} of $L$.

 Moreover, we have the following:

\begin{theorem} \cite{MN}
Let $X$ be a set with a tribracket and $L$ an oriented link diagram. 
The number $\Phi_X^{\mathbb{Z}}(L)$ of $X$-colorings of $L$ does not change
under Reidemeister moves and hence is an integer-valued invariant of oriented
links.
\end{theorem}

\begin{remark}
We can also define \textit{vertical tribrackets}, motivated by the
coloring rule
\[\includegraphics{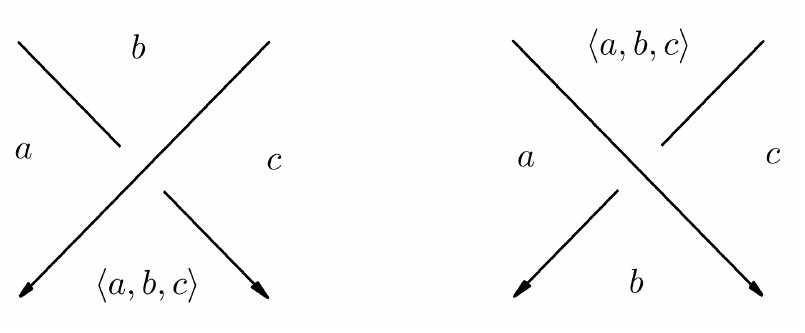}\] 
as follows:
A ternary operation $\langle,,\rangle:X\times X\times X\to X$ is a 
\textit{vertical tribracket} if it satisfies
\begin{itemize}
\item[(i)] In the equations $\langle a,b,c\rangle=d$,
any three of the variable determines the fourth, and
\item[(ii)] For all $a,b,c,d\in X$
\begin{eqnarray*}
\langle a,b,\langle b,c,d\rangle \rangle & = &
\langle a,\langle a,b,c\rangle,\langle\langle a,b,c\rangle,c,d\rangle \rangle\quad \mathrm{and}\\
\langle \langle a,b,c\rangle, c,d\rangle & = &
\langle \langle a,b,\langle b,c,d\rangle\rangle,\langle b,c,d\rangle,d\rangle\\
\end{eqnarray*}
\end{itemize}
From these coloring rules, we can observe that
given a horizontal tribracket, there is an induced vertical tribracket
(and vice-versa) satisfying
\[[a,b,\langle a,b,c\rangle]=c=\langle a,b,[a,b,c]\rangle.\]
In this paper we will work in terms of horizontal tribrackets, but all
of our results can be rephrased in terms of vertical tribrackets.
See also \cite{NOO}. 
\end{remark}

\begin{example}
Let us compute the number of colorings $\Phi_X^{\mathbb{Z}}(L2a1)$ of the 
Hopf link $L2a1$ by the Alexander tribracket structure on $X=\mathbb{Z}_5$
with $x=1$ and $y=2$, i.e. the horizontal tribracket 
$[a,b,c]=1b+2c-(1)(2)a=3a+b+2c$. 
\[\includegraphics{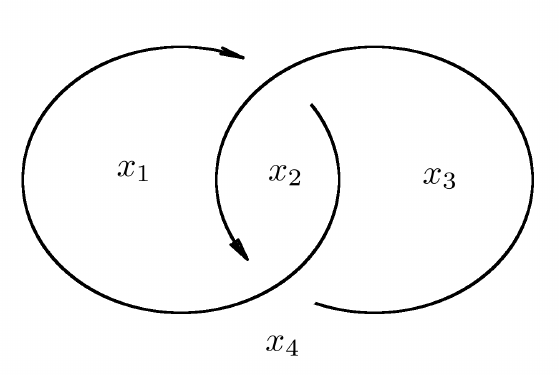}\]
We have coloring equations
\[\begin{array}{rcl}
3x_1+x_2+2x_4 & = & x_3 \\
3x_1+2x_2+x_4 & = & x_3 \\
\end{array}
\]
or in matrix form
\[
\left[\begin{array}{rrrr}
3 & 1 & 4 & 2 \\
3 & 2 & 4 & 1 \\
\end{array}\right]\leftrightarrow
\left[\begin{array}{rrrr}
1 & 0 & 3 & 1 \\
0 & 1 & 0 & 4 \\
\end{array}\right]
\]
so there are $5^2=25$ $X$-colorings of $L2a1$, including
\[\includegraphics{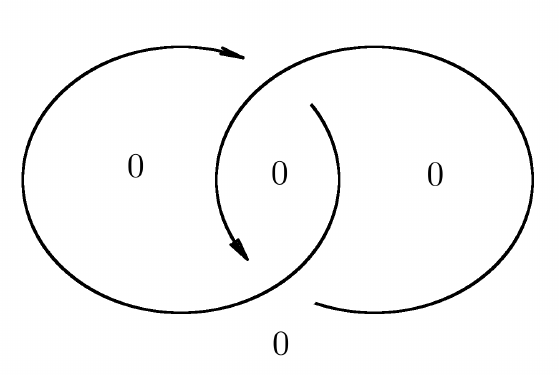}\quad\mathrm{and}\quad 
\includegraphics{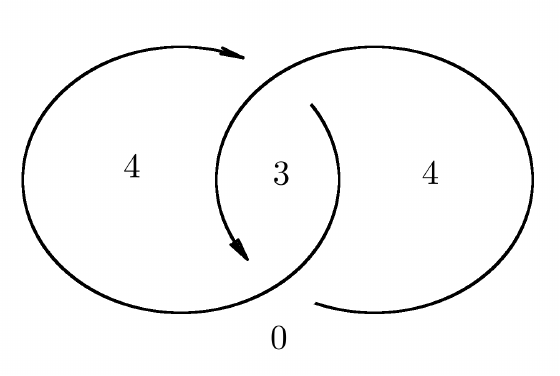}\]
etc. This invaraint thus distinguishes the
Hopf link from the unlink of two components, which has $5^3=125$ $X$-colorings.
\end{example}

\section{Multi-tribrackets}\label{M}

Let us now generalize tribrackets to the case of multi-tribrackets. 
The motivation is to have distinct tribracket operations at different
kinds of crossings with interaction laws determined by the Reidemeister
move analogues with specified crossing types we wish to allow. We begin 
with a definition inspired by conversations of the first author
with Allison Henrich and Aaron Kaestner \cite{HK}:

\begin{definition}
An \textit{explicitly typed oriented knot or link diagram} is a planar 4-valent 
directed graph with every vertex having two adjacent inputs and two adjacent 
outputs and labeled with a \textit{crossing type} chosen from a set $T$ of 
defined crossing types.
\end{definition}

Examples of crossing types include but are not limited to:
\begin{itemize}
\item \textit{Positive and negative classical crossings}
\[\includegraphics{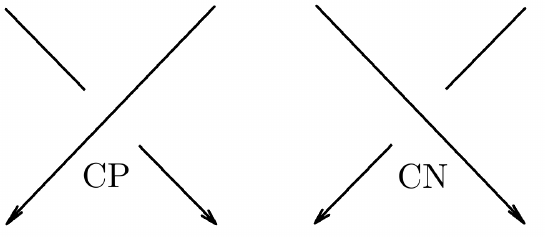}\]
\item \textit{Virtual crossings}
\[\includegraphics{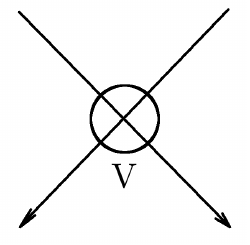}\]
\item \textit{Single-component and multi-component positive
and negative crossings}
\[\includegraphics{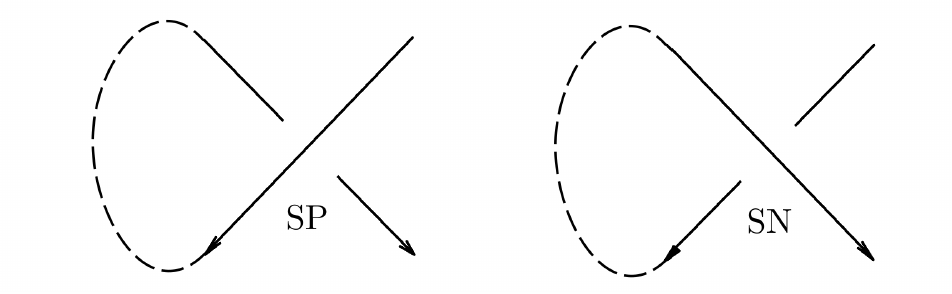}\]\[ \includegraphics{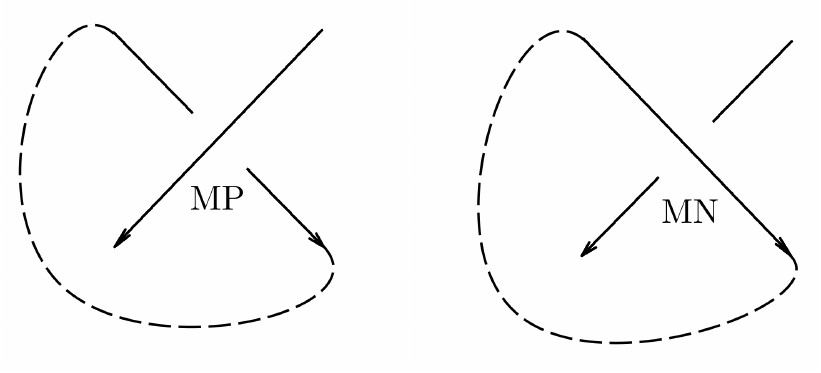}\]
\item \textit{Even and odd parity positive and negative single-component
crossings} where we have an even or odd number of crossing points between 
the over and under instances of the crossing
\[\includegraphics{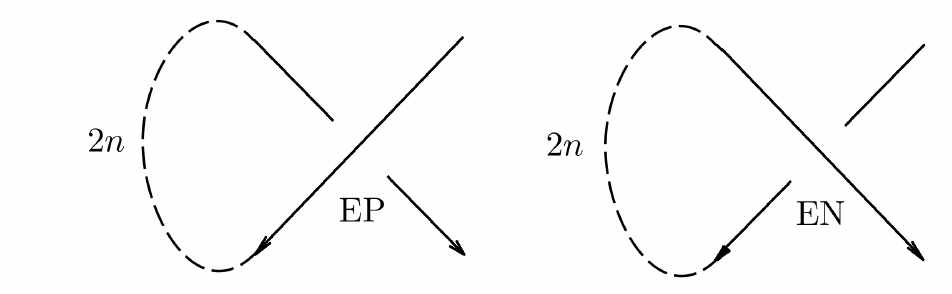}\]\[ \includegraphics{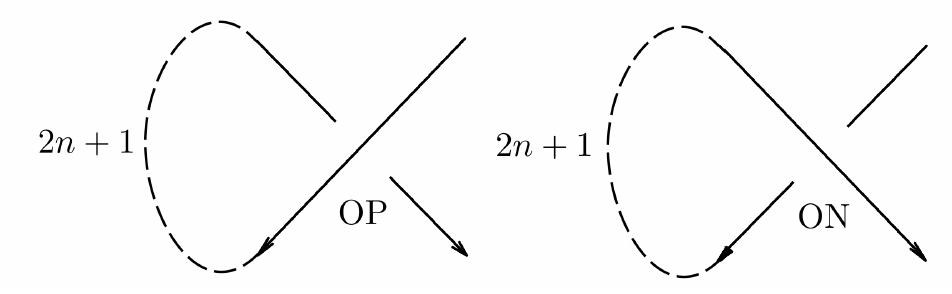}\]
etc.
\end{itemize}

Then given a set of explicit crossing types, we can select versions
of Reidemeister moves of types I, II and III with various combinations
of explicitly typed crossings to define an \textit{explicitly typed knot
theory}. The resulting equivalence classes of knot diagrams will be known 
as \textit{explicitly typed knots}.

\[\includegraphics{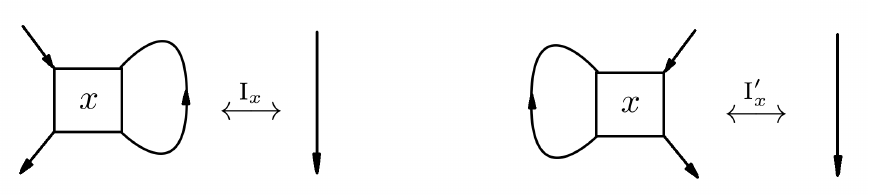}\]
\[\includegraphics{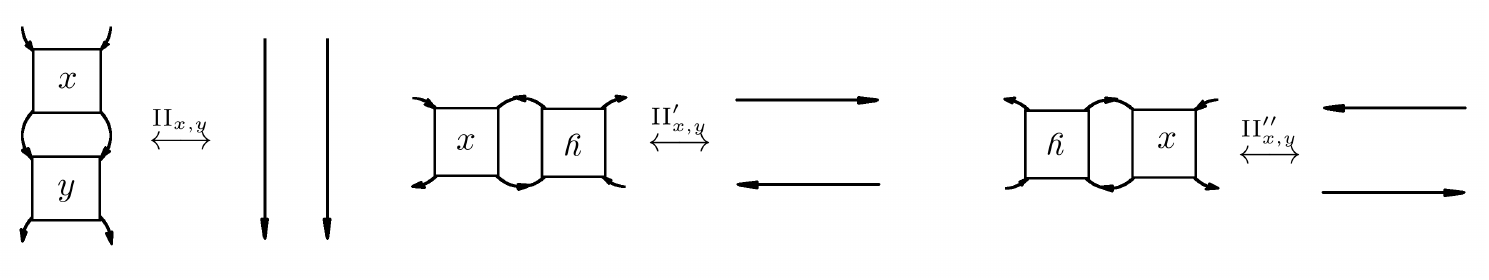}\]
\[\includegraphics{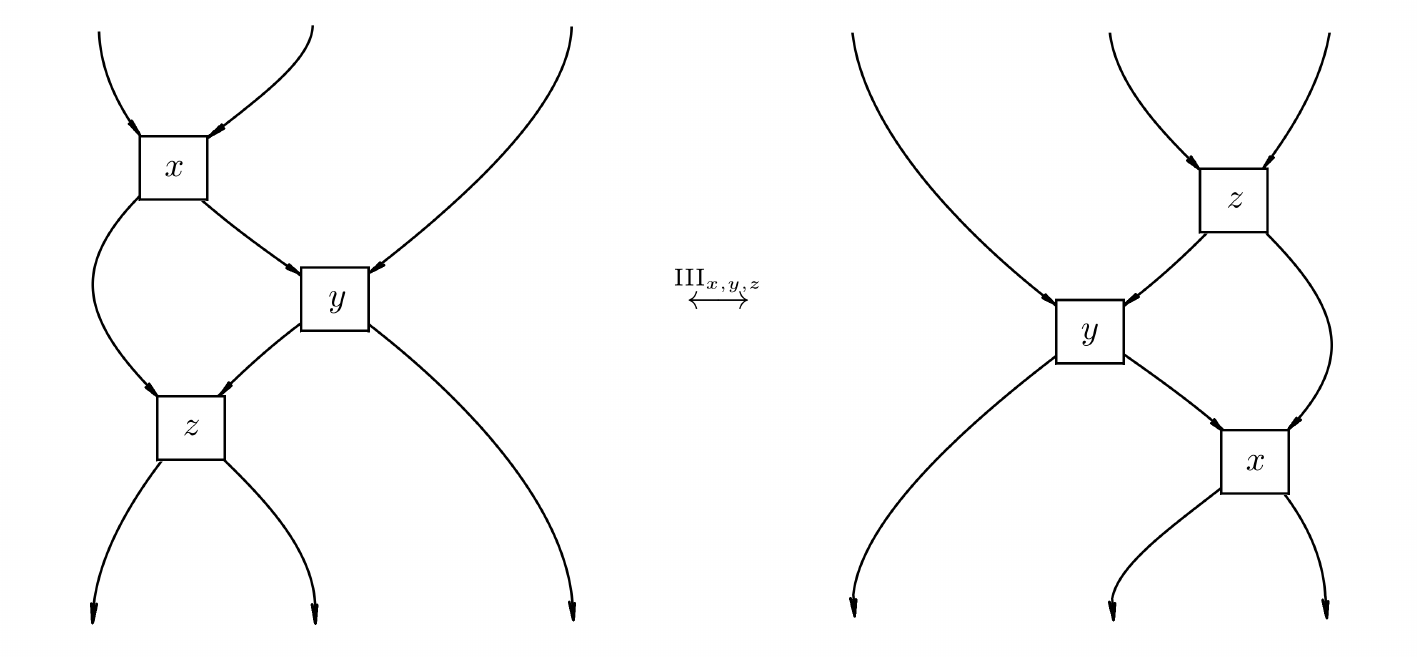}\]
\[\includegraphics{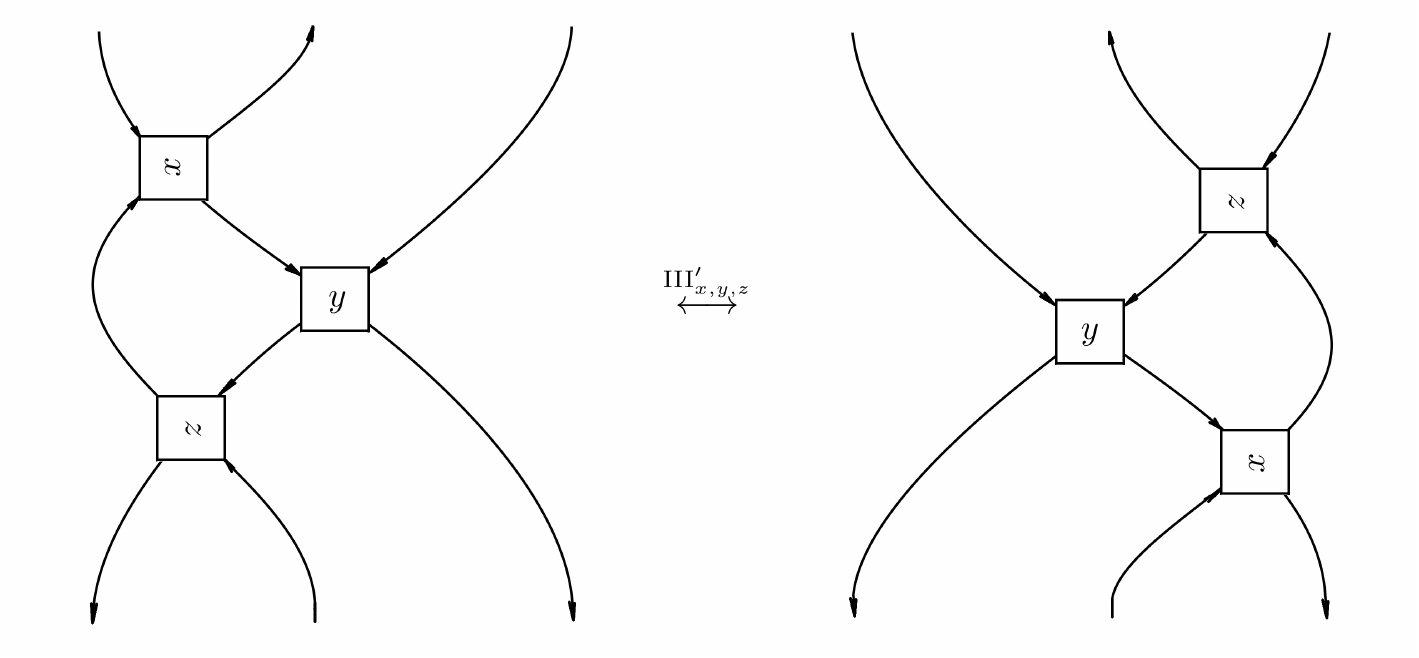}\]

Next, we make our main definition.

\begin{definition}\label{d1}
Let $X$ be a set and suppose we have an explicitly typed knot theory
with set of types $T$ and typed moves $M$. A \textit{horizontal 
multi-tribracket} structure on $X$ of type $(T,M)$ is a set of ternary 
operations \[[,,]_x:X\times X\times X\to X\]
indexed by elements $x\in T$ satisfying
\begin{itemize}
\item[(i)] For every move I$_x\in M$, we have the condition that
\[\forall a,b\in X\ \exists!\ c\in X \ \mathrm{such\ that}\ [a,b,b]_x=c,\]
\item[(i$'$)] For every move I$_x'\in M$, we have the condition that
\[\forall a,b\in X\ \exists!\ c\in X \ \mathrm{such\ that}\ [c,a,a]_x=b,\]
\item[(ii)] For every move II$_{x,y}\in M$, we have the condition that
\[\forall a,b,d\in X\ \exists!\ c\in X \ \mathrm{such\ that}\ [a,b,c]_x=[a,c,b]_y=d,\]
\item[(ii$'$)] For every move II$_{x,y}''\in M$, we have the condition that
\[\forall a,b,c\in X\ \exists!\ d\in X \ \mathrm{such\ that}\ [a,b,c]_x=[a,c,b]_y=d,\]
\item[(ii$''$)] For every move II$_{x,y}''\in M$, we have the condition that
\[\forall b,c,d\in X\ \exists!\ a\in X \ \mathrm{such\ that}\ [a,b,c]_x=[a,c,b]_y=d,\]
\item[(iii)] For every move III$_{x,y,z}\in M$,
we have for all $a,b,c,d\in X$
\[
 [b,[a,b,c]_x,[a,b,d]_y]_z 
=[c,[a,b,c]_x,[a,c,d]_z]_y 
=[d,[a,b,d]_y,[a,c,d]_z]_x.
\]
\item[(iii$'$)] For every move III$'_{x,y,z}\in M$,
we have for all $a,b,c,d\in X$
\[
 [b,[a,b,c]_x,[a,d,b]_z]_y 
=[c,[a,b,c]_x,[a,c,d]_y]_z 
=[d,[a,d,b]_z,[a,c,d]_y]_x.
\]
\end{itemize}
\end{definition}

Definition \ref{d1} is motivated by the the coloring rule
\[\includegraphics{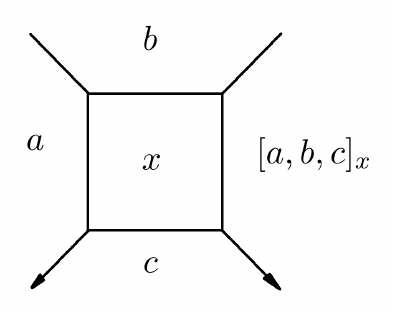}\]
applied to the explicitly typed Reidemeister moves. Specifically, the
axioms are chosen to ensure that given a valid colorings of an explicitly 
diagram on one side of a move, there is a unqiue valid coloring of the diagram
on the other side of the move which agrees with the original outside the
neighborhood of the move. Specifically, we have:


\begin{theorem}
The number of colorings of an explicitly typed link diagram by a 
multi-tribracket of type $(T,M)$ is unchanged by Reidemeister moves 
in $M$.
\end{theorem}

\begin{proof}
We consider the explicitly typed moves systematically. 

(i) We note that this one is implied by the fact that $[,,]_x$ is an operation, but we include it for completeness:
\[\includegraphics{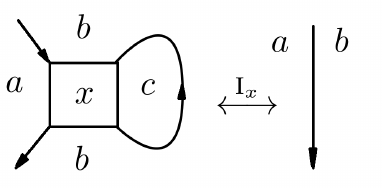}\] 
\[\forall a,b\in X\ \exists!\ c\in X \ \mathrm{such\ that}\ [a,b,b]_x=c,\]
(i$'$)
\[\includegraphics{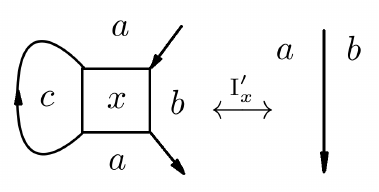}\]
\[\forall a,b\in X\ \exists!\ c\in X \ \mathrm{such\ that}\ [c,a,a]_x=b,\]
(ii)
\[\includegraphics{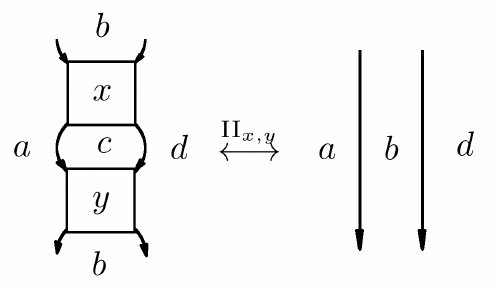}\]
\[\forall a,b,d\in X\ \exists!\ c\in X \ \mathrm{such\ that}\ [a,b,c]_x=[a,c,b]_y=d,\]
(ii$'$)
\[\includegraphics{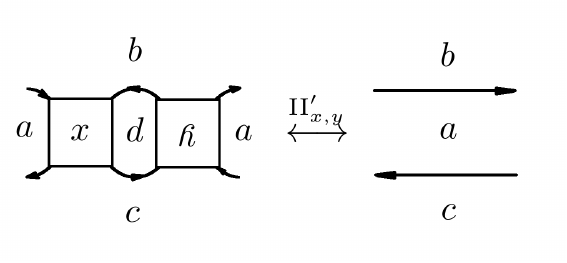}\]
\[\forall a,b,c\in X\ \exists!\ d\in X \ \mathrm{such\ that}\ [a,b,c]_x=[a,c,b]_y=d,\]
(iii)
\[\includegraphics{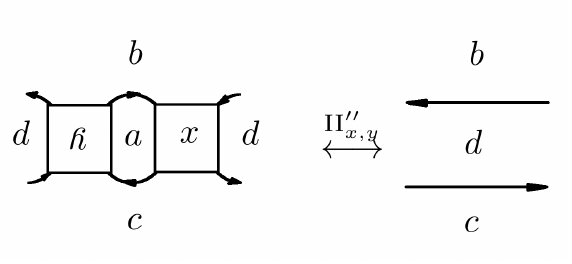}\]
\[\forall b,c,d\in X\ \exists!\ a\in X \ \mathrm{such\ that}\ [a,b,c]_x=[a,c,b]_y=d,\]
(iii$'$)
\[\includegraphics{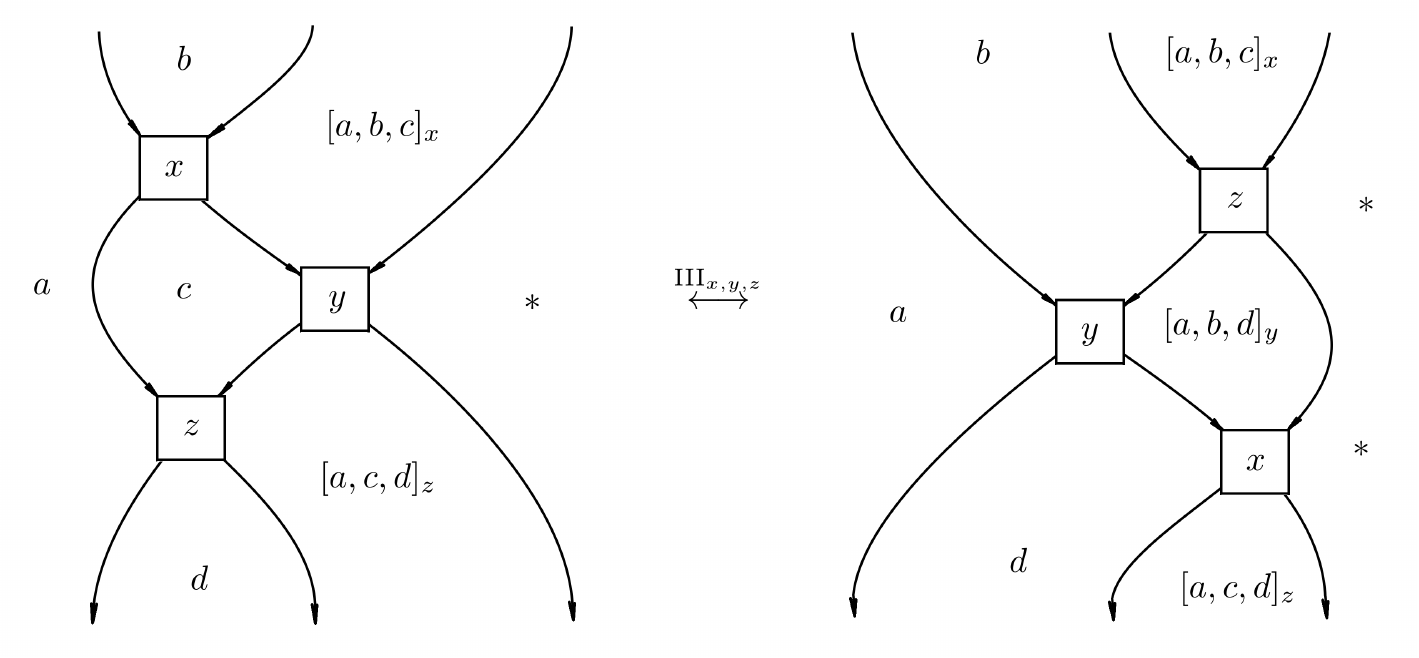}\]
\[
 [b,[a,b,c]_x,[a,b,d]_y]_z 
=[c,[a,b,c]_x,[a,c,d]_z]_y 
=[d,[a,b,d]_y,[a,c,d]_z]_x,
\]
and (iii$'$)
\[\includegraphics{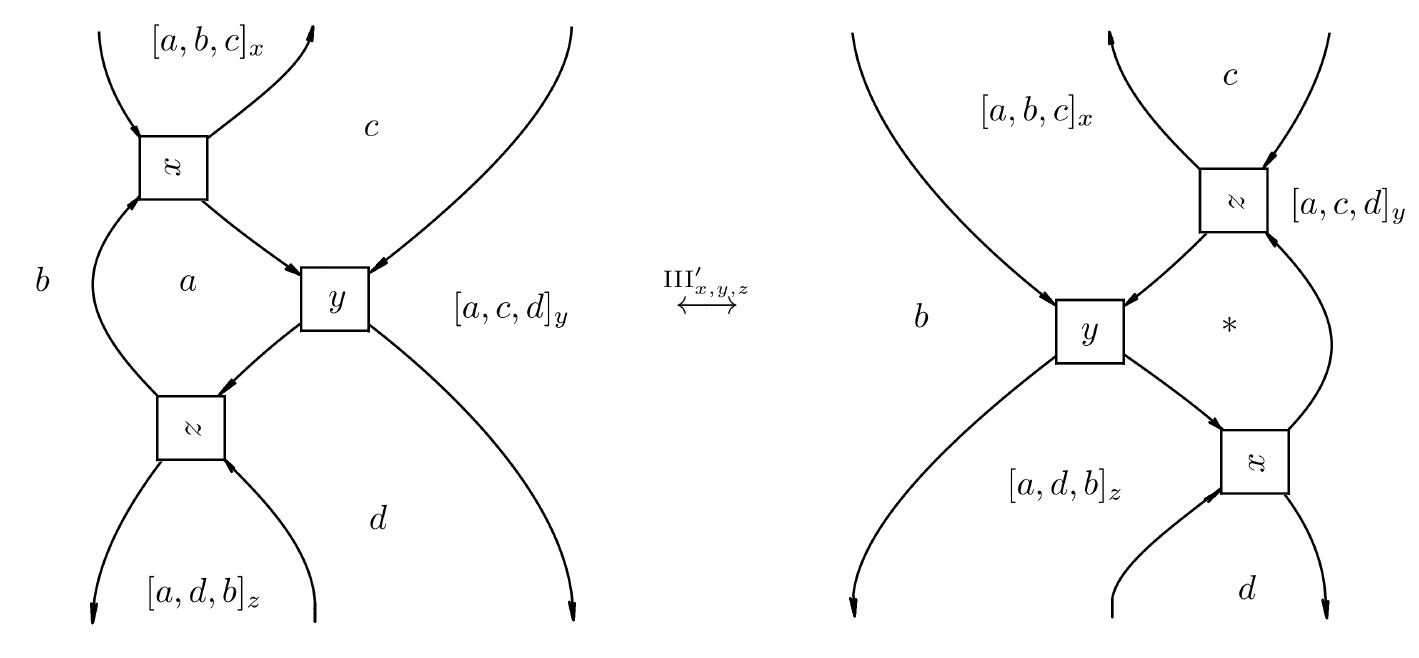}\]
\[
 [b,[a,b,c]_x,[a,d,b]_z]_y 
=[c,[a,b,c]_x,[a,c,d]_y]_z 
=[d,[a,d,b]_z,[a,c,d]_y]_x.
\]
\end{proof}

\begin{definition}
Let $X$ be a multi-tribracket of type $(T,M)$. Then the number of 
$X$-colorings of an oriented link diagram is an integer-valued invariant 
of explicitly typed knots and links of type $(T,M)$. We denote this 
invariant as $\Phi_X^{\mathbb{Z}}(L)$.
\end{definition}

\begin{example}
A tribracket is a multi-tribracket with $T=\{\mathrm{CP},\mathrm{CN}\}$ and 
\[M=\{\mathrm{I}_{\mathrm{CP}}, 
\mathrm{I}_{\mathrm{CN}},\
\mathrm{I}_{\mathrm{CP}}',\
\mathrm{I}_{\mathrm{CN}}',\
\mathrm{II}_{\mathrm{CP},\mathrm{CN}},\
\mathrm{II}_{\mathrm{CN},\mathrm{CP}},\
\mathrm{II}_{\mathrm{CP},\mathrm{CN}}',\
\mathrm{II}_{\mathrm{CN},\mathrm{CP}}',\
\mathrm{III}_{\mathrm{CP},\mathrm{CP},\mathrm{CP}}
 \}\]
where we define 
\[[a,b,c]_{\mathrm{CP}}=[a,b,c]=[a,c,b]_{\mathrm{CN}}.\]
We note that other generating sets of classical Reidemeister moves
are possible, and that this particular set is not minimal. 
\end{example}

\begin{example}
The case that originally motivated this paper was the idea of having different
tribracket operations at single-component crossings and at multicomponent 
crossings. In this case we need $T=\{SP,SN,MP,MN\}$ and
\[M=\left\{\begin{array}{l}
\mathrm{I}_{\mathrm{SP}},\ 
\mathrm{I}_{\mathrm{SN}},\
\mathrm{I}_{\mathrm{SP}}',\
\mathrm{I}_{\mathrm{SN}}',\\
\mathrm{II}_{\mathrm{SP},\mathrm{SN}},\
\mathrm{II}_{\mathrm{SN},\mathrm{SP}},\
\mathrm{II}_{\mathrm{SP},\mathrm{SN}}',\
\mathrm{II}_{\mathrm{SN},\mathrm{SP}}',\
\mathrm{II}_{\mathrm{MP},\mathrm{MN}},\
\mathrm{II}_{\mathrm{MN},\mathrm{MP}},\
\mathrm{II}_{\mathrm{MP},\mathrm{MN}}',\
\mathrm{II}_{\mathrm{MN},\mathrm{MP}}',\\
\mathrm{III}_{\mathrm{SP},\mathrm{SP},\mathrm{SP}},\
\mathrm{III}_{\mathrm{SP},\mathrm{MP},\mathrm{MP}},\
\mathrm{III}_{\mathrm{MP},\mathrm{SP},\mathrm{MP}},\
\mathrm{III}_{\mathrm{MP},\mathrm{MP},\mathrm{SP}},\
\mathrm{III}_{\mathrm{MP},\mathrm{MP},\mathrm{MP}}\end{array}
 \right\}\]
since (1) Reidemeister I moves are only single-component, (2) 
in a Reidemeister II move both crossings are single-component or both are
multi-component, and (3) in a Reidemeister III move
we can have all three crossings on the same component, one crossing
with both strands on the same component and two on different components, or 
all three crossings on different components.
\[\includegraphics{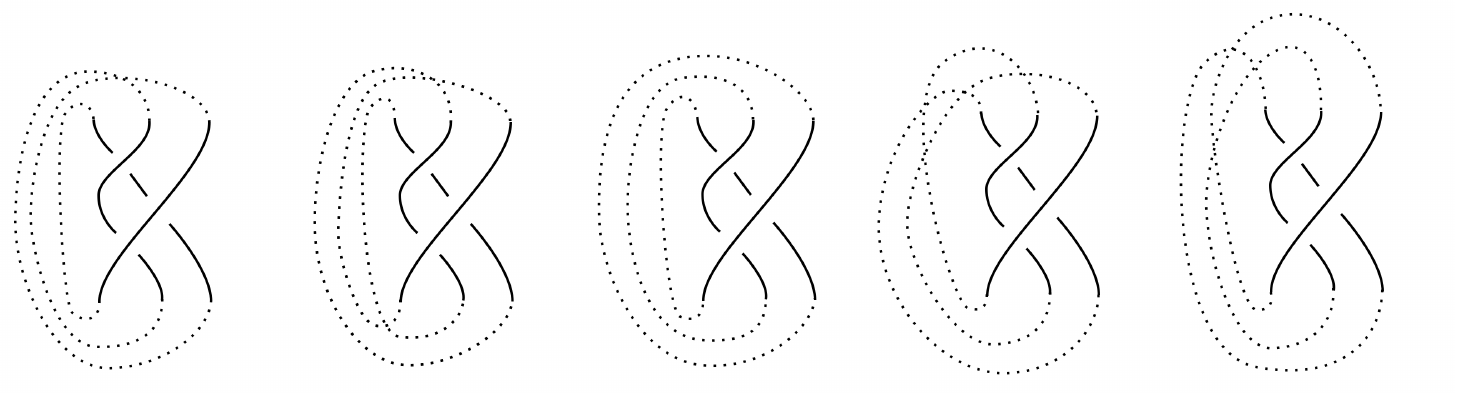}\]
As in the classical case, we can simplify this somewhat by setting
\[[a,b,c]_{\mathrm{SP}}=[a,b,c]_0=[a,c,b]_{\mathrm{SN}}.\]
\[[a,b,c]_{\mathrm{MP}}=[a,b,c]_1=[a,c,b]_{\mathrm{MN}}.\]
We call such an $X$ a \textit{multicomponent multi-tribracket}.
For such an $X$, $\Phi_X^{\mathbb{Z}}(L)$ is an invariant of oriented links which
agrees with the usual counting invariant with respect to the $X_0$ 3-tensor
for knots but can be different for links.

For instance, the pair of operation 3-tensors
\[
\left[\left[
\begin{array}{rrr}
1& 2& 3\\
3& 1& 2\\
2& 3& 1
\end{array}\right],\left[\begin{array}{rrr}
2& 3& 1\\
1& 2& 3\\
3& 1& 2
\end{array}\right],\left[\begin{array}{rrr}
3& 1& 2\\
2& 3& 1\\
1& 2& 3
\end{array}\right]\right]_0,\]\[
\left[\left[\begin{array}{rrr}
1& 3& 2\\
3& 2& 1\\
2& 1& 3
\end{array}\right],\left[\begin{array}{rrr}
3& 2& 1\\
2& 1& 3 \\
1& 3& 2
\end{array}\right],\left[\begin{array}{rrr}
2& 1& 3\\
1& 3& 2 \\
3& 2& 1
\end{array}\right]\right]_1
\]
defines a multi-tribracket whose counting invariant distinguishes the links
$L6a2$ and $L6a4$ while the single tribracket counting invariant with respect
to the single-component tribracket $X_0$  does not.
\[
\begin{array}{c}
\includegraphics{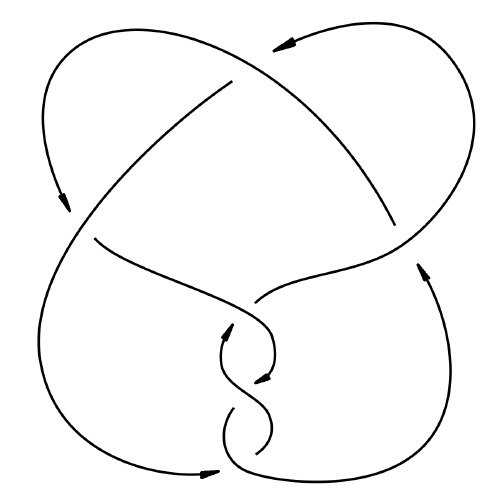} \\
\Phi_X(L6a2)=27\\
\Phi_{X_0}(L6a2)=9
\end{array}
\quad
\begin{array}{c}
\includegraphics{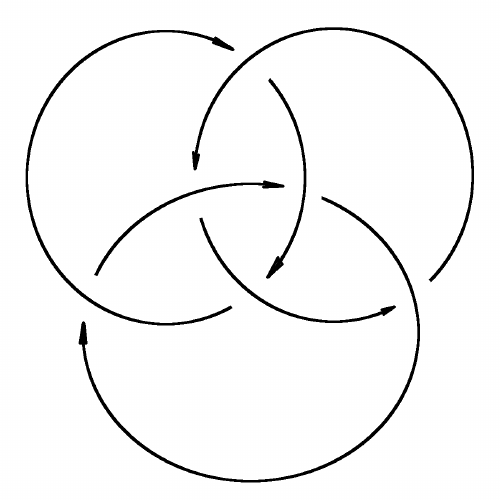} \\
\Phi_X(L6a4)=9\\
\Phi_{X_0}(L6a4)=9
\end{array}
\]
\end{example}

\begin{example}
We computed the counting invariant $\Phi_X^{\mathbb{Z}}(L)$ for all prime links
with up to 7 crossings with respect to the multicomponent multi-tribracket
\[
\left[\left[
\begin{array}{rrr}
1& 3& 2\\
2& 1& 3\\
3& 2& 1
\end{array}\right],\left[\begin{array}{rrr}
2& 1& 3\\
3& 2& 1\\
1& 3& 2
\end{array}\right],\left[\begin{array}{rrr}
3& 2& 1\\
1& 3& 2\\
2& 1& 3
\end{array}\right]\right]_0,\]\[
\left[\begin{array}{rrr}
2& 1& 3\\
1& 3& 2 \\
3& 2& 1
\end{array}\right],
\left[\left[\begin{array}{rrr}
1& 3& 2\\
3& 2& 1\\
2& 1& 3
\end{array}\right],\left[\begin{array}{rrr}
3& 2& 1\\
2& 1& 3 \\
1& 3& 2
\end{array}\right]\right]_1
\]
using our \texttt{python} code. The results are collected in the table.
\[\begin{array}{r|l}
\Phi_X^{\mathbb{Z}}(L) & L \\ \hline
9 & L5a1, L6a4, L7a1, L7a4, L7a5, L7a6  \\
27 & L2a1, L4a1, L6a1, L6a2, L6a3, L7a3\\
81 & L6a5, L6n1, L7a2, L7a7, L7n1, L7n2.
\end{array}\]

We also computed the invariant using the 4-element tribracket
\[
\left[\left[
\begin{array}{rrrr}
2& 3& 4& 1 \\
3& 4& 1& 2 \\
4& 1& 2& 3 \\
1& 2& 3& 4
\end{array}\right],\left[\begin{array}{rrrr}
1& 2& 3& 4 \\
2& 3& 4& 1 \\
3& 4& 1& 2 \\
4& 1& 2& 3
\end{array}\right],\left[\begin{array}{rrrr}
4& 1& 2& 3 \\
1& 2& 3& 4 \\
2& 3& 4& 1 \\
3& 4& 1& 2
\end{array}\right],\left[\begin{array}{rrrr}
3& 4& 1& 2 \\
4& 1& 2& 3 \\
1& 2& 3& 4 \\
2& 3& 4& 1 
\end{array}\right]\right]_0,\]\[
\left[\left[\begin{array}{rrrr}
2& 4& 1& 3 \\
1& 2& 3& 4 \\
4& 3& 2& 1 \\
3& 1& 4& 2 
\end{array}\right],\left[\begin{array}{rrrr}
4& 3& 2& 1 \\
2& 4& 1& 3 \\
3& 1& 4& 2 \\
1& 2& 3& 4
\end{array}\right],\left[\begin{array}{rrrr}
1& 2& 3& 4 \\
3& 1& 4& 2 \\
2& 4& 1& 3 \\
4& 3& 2& 1
\end{array}\right],\left[\begin{array}{rrrr}
3& 1& 4& 2 \\
4& 3& 2& 1 \\
1& 2& 3& 4 \\
2& 4& 1& 3
\end{array}\right]\right]_1
\]
using our \texttt{python} code. The results are collected in the table.
\[\begin{array}{r|l}
\Phi_X^{\mathbb{Z}}(L) & L \\ \hline
0 & L6a4 \\
16 & L5a1, L7a1, L7a3, L7a4, L7a5, L7a6 \\
32 & L2a1, L6a2, L6a3, L7n1 \\
64 & L4a1, L6a1, L6a5, L6n1, L7n1, L7a2, L7a7, L7n2.
\end{array}\]
\end{example}

\begin{example}
A \textit{virtual tribracket}, originally defined in 
\cite{NP} (where vertical tribracket notation was used), 
is a multi-tribracket with 
 $T=\{\mathrm{CP},\mathrm{CN},\mathrm{V}\}$ and 
\[M=\left\{\begin{array}{l}
\mathrm{I}_{\mathrm{CP}},\ 
\mathrm{I}_{\mathrm{CN}},\
\mathrm{I}_{\mathrm{V}},\
\mathrm{I}_{\mathrm{CP}}',\
\mathrm{I}_{\mathrm{CN}}',\
\mathrm{I}_{\mathrm{V}}',\ \\
\mathrm{II}_{\mathrm{CP},\mathrm{CN}},\
\mathrm{II}_{\mathrm{CN},\mathrm{CP}},\
\mathrm{II}_{\mathrm{V},\mathrm{V}},\
\mathrm{II}_{\mathrm{CP},\mathrm{CN}}',\
\mathrm{II}_{\mathrm{CN},\mathrm{CP}}',\
\mathrm{II}_{\mathrm{V},\mathrm{V}}',\ \\
\mathrm{III}_{\mathrm{CP},\mathrm{CP},\mathrm{CP}},\
\mathrm{III}_{\mathrm{V},\mathrm{CP},\mathrm{V}},\
\mathrm{III}_{\mathrm{V},\mathrm{V},\mathrm{V}}.\
\end{array}\right\}\]
Then for example, the pair of 3-tensors
\[
\left[\left[
\begin{array}{rrr}
1& 2& 3\\
3& 1& 2\\
2& 3& 1
\end{array}\right],\left[\begin{array}{rrr}
2& 3& 1\\
1& 2& 3\\
3& 1& 2
\end{array}\right],\left[\begin{array}{rrr}
3& 1& 2\\
2& 3& 1\\
1& 2& 3
\end{array}\right]\right]_0,\]\[
\left[\left[\begin{array}{rrr}
2& 3& 1\\
1& 2& 3\\
3& 1& 2
\end{array}\right],\left[\begin{array}{rrr}
3& 1& 2\\
2& 3& 1 \\
1& 2& 3
\end{array}\right],\left[\begin{array}{rrr}
1& 2& 3\\
3& 1& 2 \\
2& 3& 1
\end{array}\right]\right]_1
\]
defines a virtual tribracket by
\[[a,b,c]_{\mathrm{CP}}=[a,b,c]_0=[a,c,b]_{\mathrm{CN}}\]
and
\[[a,b,c]_{\mathrm{V}}=[a,b,c]_1\]
whose counting invariant distinguishes the link 
\[\includegraphics{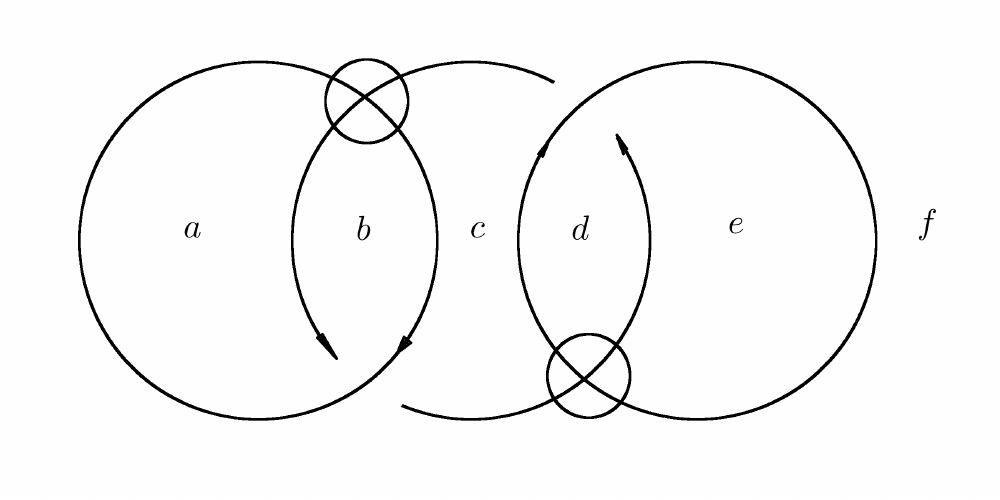}\]
from the unlink of two components, since while the latter has $3^3=27$
colorings, a coloring of the former must satisfy 
\[[a,b,f]_0=[e,d,f]_0=[a,f,b]_1=[e,f,d]_1=c=\] and it is easy to see that 
no such triple exists; hence this virtual link has no $X$-colorings.
\end{example}

\begin{example}
For our final example let us consider the case of \textit{welded links}.
Welded links can be understood as virtual links in which we are allowed to move
a classical strand over but not under a virtual crossing, thinking of the
virtual crossing as ``welded'' to the paper (see e.g. \cite{ABMW}.) 

This extra move means we can obtain invariants of welded isotopy by
taking counting invariants using \textit{welded tribrackets},
multi-tribrackets with $T=\{CP,CN,V\}$ and 
\[M=\left\{\begin{array}{l}
\mathrm{I}_{\mathrm{CP}},\
\mathrm{I}_{\mathrm{CN}},\
\mathrm{I}_{\mathrm{V}},\
\mathrm{I}_{\mathrm{CP}}',\
\mathrm{I}_{\mathrm{CN}}',\
\mathrm{I}_{\mathrm{V}}',\ \\
\mathrm{II}_{\mathrm{CP},\mathrm{CN}},\
\mathrm{II}_{\mathrm{CN},\mathrm{CP}},\
\mathrm{II}_{\mathrm{V},\mathrm{V}},\
\mathrm{II}_{\mathrm{CP},\mathrm{CN}}',\
\mathrm{II}_{\mathrm{CN},\mathrm{CP}}',\
\mathrm{II}_{\mathrm{V},\mathrm{V}}',\ \\
\mathrm{III}_{\mathrm{CP},\mathrm{CP},\mathrm{CP}},\
\mathrm{III}_{\mathrm{CP},\mathrm{CP},\mathrm{V}},\
\mathrm{III}_{\mathrm{V},\mathrm{CP},\mathrm{V}},\
\mathrm{III}_{\mathrm{V},\mathrm{V},\mathrm{V}}.\
\end{array}\right\}\]
Our python code reveals welded tribrackets including for example
\[\scalebox{0.95}{$
\left[\left[\begin{array}{rrrrr}
4 &1 &2 &5 &3 \\
1 &2 &3 &4 &5 \\
2 &3 &5 &1 &4 \\
5 &4 &1 &3 &2 \\
3 &5 &4 &2 &1
\end{array}\right],\left[\begin{array}{rrrrr}
5 &4 &1 &3 &2 \\
4 &1 &2 &5 &3 \\
1 &2 &3 &4 &5 \\
3 &5 &4 &2 &1 \\
2 &3 &5 &1 &4
\end{array}\right],\left[\begin{array}{rrrrr}
3 &5 &4 &2 &1 \\
5 &4 &1 &3 &2 \\
4 &1 &2 &5 &3 \\
2 &3 &5 &1 &4 \\
1 &2 &3 &4 &5
\end{array}\right],\left[\begin{array}{rrrrr}
1 &2 &3 &4 &5 \\
2 &3 &5 &1 &4 \\
3 &5 &4 &2 &1 \\
4 &1 &2 &5 &3 \\
5 &4 &1 &3 &2
\end{array}\right],\left[\begin{array}{rrrrr}
2 &3 &5 &1 &4 \\
3 &5 &4 &2 &1 \\
5 &4 &1 &3 &2 \\
1 &2 &3 &4 &5 \\
4 &1 &2 &5 &3
\end{array}\right]\right]_0$}
\]
\[\scalebox{0.95}{$
\left[\left[\begin{array}{rrrrr}
2 &3 &5 &1 &4 \\
3 &5 &4 &2 &1 \\
5 &4 &1 &3 &2 \\
1 &2 &3 &4 &5 \\
4 &1 &2 &5 &3
\end{array}\right],\left[\begin{array}{rrrrr}
1 &2 &3 &4 &5 \\
2 &3 &5 &1 &4 \\
3 &5 &4 &2 &1 \\
4 &1 &2 &5 &3 \\
5 &4 &1 &3 &2
\end{array}\right],\left[\begin{array}{rrrrr}
4 &1 &2 &5 &3 \\
1 &2 &3 &4 &5 \\
2 &3 &5 &1 &4 \\
5 &4 &1 &3 &2 \\
3 &5 &4 &2 &1
\end{array}\right],\left[\begin{array}{rrrrr}
3 &5 &4 &2 &1 \\
5 &4 &1 &3 &2 \\
4 &1 &2 &5 &3 \\
2 &3 &5 &1 &4 \\
1 &2 &3 &4 &5
\end{array}\right],\left[\begin{array}{rrrrr}
5 &4 &1 &3 &2 \\
4 &1 &2 &5 &3 \\
1 &2 &3 &4 &5 \\
3 &5 &4 &2 &1 \\
2 &3 &5 &1 &4
\end{array}\right]\right]_1.$}
\]
It then follows that the number of colorings of a welded link diagram by 
a welded tribracket is an invariant of welded links.

\end{example}

\section{Questions}\label{Q}

We conclude in this section with several possibilities for future directions
of research. 

\begin{itemize}
\item \textit{Parity multi-tribrackets}. For any crossing in a knot or link
diagram, the \textit{parity} (even or odd) of the number of crossing points
one encounters on the journey from the overcrossing point to the undercrossing
point is unchanged by Reidemeister moves. This has led to the notion of 
\textit{parity biquandles} in \cite{KK}, which are effectively just usual 
biquandles when applied to classical knots but yield stronger invariants
for virtual links. Parity multi-tribrackets are more cumbersome 
because of the necessity of including an operation at virtual crossings,
but a parity virtual tribracket can be descrbed with a triple of compatible 
3-tensors defining operations at even, odd and virtual crossings.
\item \textit{Enhancements}. As with quandles, biquandles and other coloring
structures, we can define enhancements of the multi-tribracket counting 
invariant, e.g. extending the tribracket modules in \cite{KNS} to the case of
multi-tribrackets.
\item \textit{Homology}. How can we modify the tribracket homology found in
\cite{MN4,NOO} to adapt to the case of multi-tribrackets?
\item \textit{Multi-cocycle enhancements}. Following \cite{CN,KNS}, define
a general theory of multi-tribracket cocycle enhancements.
\end{itemize}

\bibliography{sn-ep}{}
\bibliographystyle{abbrv}

\noindent
\textsc{Department of Mathematical Sciences \\
Claremont McKenna College \\
850 Columbia Ave. \\
Claremont, CA 91711}

\end{document}